\documentclass[11pt]{amsart}
\usepackage{amsmath,amssymb,amsthm}
\usepackage{amsrefs}
\usepackage{hyperref}

  \newtheorem{lemma}{Lemma}[section]
\newtheorem{thm}[lemma]{Theorem}

\newtheorem{prop}[lemma]{Proposition}
\newtheorem{conjecture}[lemma]{Conjecture}
\theoremstyle{definition}
\newtheorem{definition}[lemma]{Definition}
\newtheorem{remark}[lemma]{Remark}
\newtheorem{example}[lemma]{Example}
\theoremstyle{remark}

  \newcommand{\bC}{\mathbb{C}}
    \newcommand{\bR}{\mathbb{R}}
      \newcommand{\bP}{\mathbb{P}}
  \newcommand{\A}{\mathcal{A}}
   \newcommand{\B}{\mathcal{B}}
   \newcommand{\PP}{{\mathbb P}}
\DeclareMathOperator{\ind}{ind}

\begin{document}
\title{Confocal Families of Plane Algebraic Curves}

\author[R.~Piene]{Ragni Piene}
\address{Ragni Piene\\Department of Mathematics\\
University of Oslo\\P.O.Box 1053 Blindern\\NO-0316 Oslo\\Norway}
\email{\href{mailto:ragnip@math.uio.no}{ragnip@math.uio.no}}
\urladdr{\href{http://www.mn.uio.no/math/english/people/aca/ragnip/index.html}
{www.mn.uio.no/math/english/people/aca/ragnip/index.html}}

\author[B. Shapiro]{Boris Shapiro}
\address{Boris Shapiro\\Department of Mathematics\\Stockholm University\\SE-10691\\Stockholm\\Sweden}
\email{\href{mailto:shapiro@math.su.se}{shapiro@math.su.se}}
\urladdr{\href{https://staff.math.su.se/shapiro}
{https://staff.math.su.se/shapiro}}

\date{\today}

\begin{abstract}
We study real plane algebraic curves with prescribed focal divisors.  Confocality is reformulated by means of a focal map on real equiclassical families of dual curves.  The differential of this map is governed by a focal adjoint line bundle on the normalization of the dual curve; this bundle is obtained from the class adjoint bundle by subtracting one hyperplane section.  As an application, we prove a maximal-rank statement for rational nodal-cuspidal dual curves and obtain the expected local dimension of their confocal loci.  We also compare this deformation-theoretic picture with several classical constructions of curves with prescribed foci.
\end{abstract}



\maketitle

\centerline{\emph{To the late Vladimir Arnold on the occasion of his 90th birthday}}

\section{Introduction}

The main goal of this paper is to study families of real plane algebraic curves admitting confocal deformations, that is, positive-dimensional families of curves with the same focal divisor.  Equivalently, for curves with simple real foci, this means families sharing the same set of real foci.

In the nineteenth century, foci of ellipses were generalized to foci of higher-degree curves: the (real) \emph{foci} of a real algebraic curve $C \subset \bP^2$ are the real points on isotropic tangent lines to $C$, see \cite{Sa}*{p.~120}. In traditional terminology associated with an algebraic curve $F(x,y,z)=0$ in the complex projective plane $\bP^2$, an isotropic line is one containing either of the circular points $p_\pm=(1:\pm i :0)$ (the pair of points common to all circles). Although this notion is not especially intuitive at first sight, it appears in a number of research areas including approximation theory, numerical analysis, and fluid models. This is the main reason for discussing foci of real algebraic curves again.

\subsection*{Short historical account}
\label{sec:int}
 The points now called foci first arose in connection with conic sections. According to the classical book \cite{Go}*{p.~252} on the history of mathematics, Apollonius of Perga, who lived in the third century B.C., incidentally discovered the foci of the ellipse and hyperbola. Five hundred years later Pappus of Alexandria found the focus of the parabola. The theory of foci was first worked out systematically by J.~Kepler while formulating his famous laws of planetary motion in the early seventeenth century; he also introduced the term \emph{focus}. The general notion of a focus of an algebraic curve of degree higher than two does not seem to have been established until 1832 by J.~Pl\"ucker, see \cite{Pl}*{p.~85}.

A general discussion of foci and their properties may be found in \cite{An, Sa, Ba, Wi, Hi, Co, Ri, Ich}.  Hilton and Jervis describe graphically the foci of several cubics and quartics, see \cite{HiJe}. 
These works study both the geometry of foci of individual curves and inverse problems with prescribed foci.
Roberts studies foci and confocal systems of various plane curves of degree $\le 4$, see \cite{Roberts1901,Roberts1904}. Jeffrey considers cubics of class three with prescribed focal configurations and cubics with a double and a simple focus, see \cite{Je,Je2}. Emch studies curves with a prescribed system of foci and computes the foci of a special cubic, see \cite{Emch}.

More recently, the subject has been revisited from different perspectives, including geometric approaches to real foci
(Langer--Singer \cite{LS})  and connections with complex polynomials (Casas-Alvero \cite{CA1, CA2}), as well as reconstruction results in special
families (Poncelet--Darboux \cite{Hunziker} and Kippenhahn-type curves \cite{Ki}).

However, the deformation-theoretic problem of describing families of curves with the same focal divisor does not appear to have been systematically studied. In particular, we address the following question:

\begin{quote}
Which plane algebraic curves admit nontrivial confocal deformations?
\end{quote}

The main conceptual contribution of this paper is the introduction of a focal map on equiclassical families and the interpretation of confocal families as its fibers. This allows us to study confocality using deformation theory.  The key point is that, over the reals, the infinitesimal condition of fixing the focal divisor is  the condition of vanishing  on the union of the two conjugate isotropic lines
\[
Y=\{u^2+v^2=0\}=L_+\cup L_-.
\]
Consequently the relevant cohomological object is the focal adjoint bundle
\[
\B_D=\A_D\otimes \nu^*\mathcal O_D(-1)
\]
on the normalization of the dual curve.  Under the usual equiclassical tangent-space hypotheses, the kernel of the differential of the focal map is identified with
\[
H^0(\widetilde D,\B_D).
\]
For rational nodal-cuspidal dual curves this gives the maximal-rank formula
\[
\operatorname{rank}(d\Phi_G)=\min(2c,c+d+1),
\]
and the expected local dimension of the confocal locus is
\[
\max(0,d-c+1).
\]
The general maximal-rank statement is formulated below as a Brill--Noether type conjecture for \(\B_D\).

\medskip

The organization of the paper is as follows. Section~2 collects the classical preliminaries on isotropic lines, foci, dual curves, and focal divisors. Section~3 introduces the focal map on equiclassical families, identifies the infinitesimal kernel with the cohomology of the focal adjoint bundle, proves the rational maximal-rank statement, and formulates the main Brill--Noether type conjecture. Section~4 discusses the principal known examples and non-examples of confocal families, including maximal-class rational curves, curves of minimal class, Siebeck curves, and Poncelet curves.

The appendix, written by Eugenii Shustin from Tel Aviv University after reading a preliminary version of this text, contains additional comments on Conjecture~\ref{conj:main} and proves it in a range not covered by Proposition~\ref{prop:BN-reformulation}.

\section{Classical preliminaries}

Throughout the paper,  a \emph{tangent} to curve at a point $p$ is a line whose intersection number with the curve at $p$ is strictly greater than the multiplicity of $p$ as a point on the curve. Hence, the tangents are the lines corresponding to the points on the dual curve.
In particular, not every line passing through a singular point is considered to be a  tangent. 

It is well-known that every circle intersects the line at infinity in two imaginary points $p_\pm=(1:\pm i:0)$. These points are called the \emph{circular points at infinity}.  
In the case of a \emph{real} curve, the two circular points at infinity are related to the curve by an analogous projective construction, namely via isotropic lines and tangents through these points.  Unlike the case of a circle, the curve need not pass through the circular points; the analogy is only at the level of the projective construction. Lines passing through the circular points, other than the line at infinity, are called \emph{circular lines} or \emph{isotropic lines}. 

If from a focus of a conic we draw two tangents to the curve, these pass respectively through the two circular points at infinity. This fact led Pl\"ucker \cite {Pl}*{p.~85} to his generalization of the foci of curves of higher degree. His definition, as stated by Cayley in \cite {Cay}*{p.~515}, is as follows: ``If from each of the circular points at infinity \ldots tangents are drawn to the curve, the intersections of each tangent from the one point with each tangent from the other point are the foci of the curve".

This definition not only gives real foci but also imaginary foci corresponding to the imaginary intersections of the tangents from the circular points to the curve.\footnote{With this definition the ellipse, for instance, has four foci, two real and two imaginary.}

If the circular points $p_\pm$ do not lie on the curve, the finite intersections of ordinary tangents (tangents having ordinary contact) from $p_+$ and $p_-$ are called \emph{ordinary foci}. If $p_+$ and $p_-$ are on the curve, the finite intersections of the tangents to the curve at $p_+$ and $p_-$ are called \emph{singular foci} of the curve. The intersections of the remaining tangents from $p_+$ and $p_-$ to the curve are ordinary foci. If one of the two isotropic tangents is a tangent at a circular point lying on the curve and the other is not, then their intersection is counted among the ordinary foci.  The term singular focus is reserved for the intersection points arising from tangents at the circular points themselves.

\medskip
The \emph{class} of a plane curve is the number of tangents that can be drawn to the curve from a general point. Equivalently, the class is the degree of the dual curve. Consider a real curve of class $c$. Then $c$ tangents can typically be drawn to the curve from $p_+$ and $p_-$. These isotropic tangents will intersect in $c^2$ points. If such a tangent from $p_+$ is given by $x+iy=(a+ib)z$, then $x-iy=(a-ib)z$ is a tangent from $p_-$. These two tangents intersect in a real point $(a:b:1)$, so there are $c$ real foci. There are no more than $c$ real foci, for no tangent from $p_\pm$ can contain more than one real point. Hence a curve of class $c$ has in general $c$ real and $c^2-c$ unreal foci \cite{Hi}*{p.~69}. 

The number $c$ of real foci will be decreased if the curve is tangent  to the line at infinity or passes through the circular points.  Another possible drop occurs when a double tangent of the curve passes through one of the circular points. (More information how the number of real foci decreases can be found in  \cite{Ri}.)  

\section{Which degree-$d$ curves can admit confocal families?}
We start with the consideration of general plane curves of given degree $d$. 
The complete linear system of plane curves of degree $d$ has dimension
\[
N_d=\dim \mathbb P H^0(\mathbb P^2,\mathcal O_{\mathbb P^2}(d))=\frac{d(d+3)}{2}.
\]
For a given smooth curve of degree $d$, the class is $c=d(d-1)$. 
If $G(u,v,w)=0$ is the equation of the dual curve,  the focal data are encoded by the two degree-$c$ polynomials
\[
G_+(w):=G(-1,-i,w) \quad \text{and} \quad G_-(w):=G(-1,i,w).
\]
Indeed, if $x+iy=r_+z$ (resp.  $x-iy=r_-z$) is the equation of an isotropic line through $p_+$ (resp. $p_-$), then the line is tangent to the curve iff $r_+$ (resp. $r_-$) is a root of $G_+$ (resp. $G_-$).
The intersection
\[\left(\frac{1}2 (r_++r_-):-\frac{i}2(r_+-r_-):1\right)
\]
of two such lines is a focus of the curve. To ask that a given point $(f_+:f_-:1)$ be a focus, amounts to fixing one root $r_+$ of $G_+$ and one root $r_-$ of $G_-$:
\[r_+=f_+ +if_- \text{ and } r_-=f_+ - if_-.\]
Hence, to ask for a fixed focus, amounts to imposing two conditions on the dual curve, hence on the original curve.

The focus is real iff the two roots $r_+$ and $r_-$ are complex conjugate. 
If the curve is real, then also its dual is real, so that $G$ can be chosen to have real coefficients. In this case, $G_-$ is the complex conjugate of $G_+$, so if $r_+$ is a root of $G_+$, then $r_-:=\overline{r_+}$ is a root of $G_-$, and hence the corresponding focus is a real point. More precisely, it follows from the above that if $r_+=a+ib$ is a root of $G_+$, then $(a:b:1)$ is a real focus of the curve.

It follows that a general real curve of degree $d$ has exactly $c=d(d-1)$ finite real foci, counted with multiplicity. Indeed, the isotropic tangents from $p_+$ are finite and simple for a general curve, and their conjugates from $p_-$ give the corresponding real points. It cannot have more, since two isotropic lines that are not conjugate cannot intersect in a real point, and each isotropic line contains exactly one real point.

A naive dimension count then gives that the set of curves with $c$ fixed foci should have dimension
\[
N_d-2c=\frac{d(d+3)}{2}-2d(d-1)=-\frac{1}{2}d(3d-7)
\]
since fixing \(c\) foci is expected to impose \(2c\) real conditions.
Now  $N_d-2c$ is negative for $d\ge 3$. So we should expect that a general smooth curve of degree $d\ge 3$ does not belong to a positive-dimensional confocal family; among smooth curves, only conics have room for such a family.

More generally, let $V_{d,g,c}$ denote the equiclassical family of irreducible plane curves of degree $d$, genus $g$, and class $c$ \cite{Sh98}*{p.~196}. Its expected dimension is
\[
d-g+c+1.
\]
Fixing $c$ foci should again impose $2c$ conditions, so the expected dimension of the confocal locus is
\[
d-g+c+1-2c=d-g-c+1.
\]
This suggests that, generically, positive-dimensional confocal families can occur only when the inequality 
\[
c\le d-g
\]
holds. Very exceptional positive-dimensional families are not excluded by this dimension count. 
In particular, curves admitting confocal families should be very special singular curves of unusually small class.
\medskip

Shustin proved that if $2g\le d+1$, then $V_{d,g,c}$ is an irreducible variety of dimension $d-g+c+1$, and a general point corresponds to a curve with only nodes and cusps as singularities \cite{Sh98}*{Cor.~1.2, p.197}. A classical example where irreducibility fails, was given by Zariski:  $V_{6,4,12}$ consists of two $15$-dimensional irreducible families of six-cuspidal sextics  \cite{Za}*{p.~223}.

\begin{remark}
Thus the basic existence problem with prescribed real foci is naturally a problem about class rather than degree: every $c$-tuple of distinct real points occurs as the set of real foci of some real curve of class $c$, while the rational case considered below is much more rigid.
\end{remark}

The rational case is the most favorable one. If $g=0$, the inequality becomes $c\le d$. Thus rational curves are the natural first case to study systematically, although in principle one could also look for singular curves of higher genus with exceptionally small class.

A particularly transparent test occurs for rational nodal-cuspidal curves. Suppose $C$ is a rational plane curve of degree $d$ with $\delta$ nodes and $\kappa$ cusps. Then
\[
\delta+\kappa=\frac{(d-1)(d-2)}{2},
\]
and the Pl\"ucker formula gives
\[
c=d(d-1)-2\delta-3\kappa=2(d-1)-\kappa.
\]
Hence the expected dimension of the confocal locus becomes
\[
d+1-c=d+1-(2d-2-\kappa)=3-d+\kappa.
\]
Therefore a rational nodal-cuspidal curve can be expected to admit a positive-dimensional confocal family only if
\[
\kappa\ge d-2.
\]
For instance, nodal rational curves ($\kappa=0$) should give a positive-dimensional family only for $d=2$, a finite confocal locus for $d=3$, and no positive-dimensional confocal family for $d\ge 4$. Each cusp increases the expected confocal dimension by one.

The preceding paragraph concerns the singularities of the original curve $C$.  In the deformation-theoretic result below we work instead with the dual curve
$D=C^\vee\subset (\mathbb P^2)^\vee$ and impose the nodal-cuspidal hypothesis on $D$.  Although $C$ and $D$ have the same normalization in the birational duality situation, the condition of being nodal-cuspidal is not literally self-dual and should be read on the curve under discussion.

Let $\mathbb P(H^0((\mathbb P^2)^\vee,\mathcal O_{(\mathbb P^2)^\vee}(c))\cong \mathbb P^{\frac{1}2 c(c+3)}$ denote the set of curves in $(\bP^2)^\vee$ of degree $c$, with coordinates given by the coefficients of the polynomials. Let $W_{c,g,d}\subset \mathbb P^{\frac{1}2 c(c+3)}$ denote the equiclassical family of curves of degree $c$, genus $g$, and class $d$. We shall work on the open subset on which the coefficient of $w^c$ in $G$ is nonzero. After dividing by this coefficient, we are on the affine chart where the coefficient of $w^c$ is $1$; in particular, the focal divisor considered below is finite in this affine chart.  On this open subset consider the map
\[
\Phi\colon W_{c,g,d} \to \mathbb C[w]_{\le c-1}\cong \bC^{c}, \quad
\Phi(G):=G_+(w)-w^c,
\]
where
\[
G_+(w)=G(-1,-i,w).
\]
To fix the coefficients of $G_+$ is equivalent to fixing its finite roots, counted with multiplicity.  We call the divisor of roots of $G_+$ the \emph{focal divisor} in the chosen affine chart.  For a real curve this divisor is conjugate to the corresponding divisor of roots of $G_-$, and it records the finite real foci counted with multiplicity.

\begin{definition}
Let $W_{c,g,d}(\bR)\subset W_{c,g,d}$ be the subset of real polynomials. The \emph{focal map} is the restriction
\[
\Phi\colon W_{c,g,d}(\bR)\to \bR^{2c},
\]
where $\mathbb C[w]_{\le c-1}$ is regarded as a real vector space.
\end{definition}
If $G$ is real, then $G_-(w):=G(-1,i,w)$ is the complex conjugate of $G_+(w)$. Thus each root of $G_+(w)$ gives one real focus, counted with the same multiplicity as a root of the focal divisor.  A fiber of $\Phi$ therefore consists of dual curves whose corresponding real curves have fixed focal divisor; when all roots are simple, this is the same as fixing the set of real foci.

It will be important below that the real differential of $\Phi$ is controlled not by one isotropic line alone, but by the union of the two conjugate isotropic lines
\[
Y=L_+\cup L_-,\qquad L_+=\{u+iv=0\},\quad L_- =\{u-iv=0\}.
\]
Indeed, for a real infinitesimal deformation $H$ of $G$, the condition that the first variation of $G_+$ vanishes is equivalent to the simultaneous vanishing of $H$ on $L_+$ and on $L_-$. Equivalently, after complexification of the real tangent map, one is restricting to the reducible conic
\[
Y=\{u^2+v^2=0\}.
\]
This elementary point is the source of the shift by two, rather than by one, in the adjoint bundles appearing below.

\begin{prop}\label{prop:focal-fibers}
Let $W\subseteq W_{c,g,d}(\bR)$ be an irreducible component. The locus in $W$ of curves with prescribed focal divisor is a fiber of $\Phi$. Hence, at a smooth point $G\in W$, the Zariski tangent space to the local confocal locus is
\[
\ker(d\Phi_G).
\]
If $W$ has the expected real dimension $c+d-g+1$ and $d\Phi_G$ has maximal rank, then
\[
\operatorname{rank}(d\Phi_G)=\min(2c,c+d-g+1)
\]
and the expected local dimension of the confocal locus through $G$ is
\[
\max(0,d-g-c+1).
\]
In particular, positive-dimensional local confocal deformations are expected only when
\[
c\le d-g.
\]
\end{prop}

\begin{proof}
Two real curves have the same focal divisor exactly when their dual equations have the same polynomial $G_+(w)$, up to the normalization used above, because the roots of $G_+(w)$ are precisely the numbers $x+iy$ corresponding to the real foci $(x:y:1)$, counted with multiplicity. Thus the locus of curves with prescribed focal divisor is a fiber of $\Phi$.

At a smooth point the tangent space to the fiber is the kernel of the differential. If the differential has maximal rank, its rank is the smaller of the dimension of the source and the dimension $2c$ of the target. Since the expected dimension of $W_{c,g,d}$ is $c+d-g+1$, the asserted formula follows.
\end{proof}

We now introduce the adjoint line bundles which control this differential.  The first one is a class adjoint bundle.  The second, obtained from it by subtracting one hyperplane section, is the bundle relevant for confocal deformations.

We shall use the following standard deformation-theoretic convention.  At a point $D=\{G=0\}$ of an equiclassical family, the phrase \emph{the usual equiclassical tangent-space description holds} means that the complex Zariski tangent space to the equiclassical stratum is represented by
\[
T_DW_{\mathbb C}\simeq H^0(\mathcal I_{Z^{\rm ec}}(c))/\langle G\rangle,
\]
where $Z^{\rm ec}$ is the equiclassical scheme of $D$.  On the affine chart where the coefficient of $w^c$ is normalized to be $1$, this is represented by sections of $H^0(\mathcal I_{Z^{\rm ec}}(c))$ with zero $w^c$-coefficient.

\begin{definition}\label{def:class-adjoint-bundle}
Let $D\subset (\bP^2)^\vee$ be a reduced curve of degree $c$, let $\nu:\widetilde D\to D$ be the normalization, and let $Z^{\rm ec}$ be an equiclassical scheme of $D$ for which the usual tangent-space description holds. Assume also that the equiclassical adjoint sequence holds for $D$. The \emph{class adjoint line bundle} $\A_D$ is the line bundle on $\widetilde D$ defined by
\[
0\to \mathcal O_{(\bP^2)^\vee}(-1)
\to \mathcal I_{Z^{\rm ec}}(c-1)
\to \nu_*\A_D\to 0.
\]
Equivalently, $\A_D$ is obtained by pulling the torsion-free sheaf $\mathcal I_{Z^{\rm ec}}(c-1)\otimes \mathcal O_D$ to the normalization and removing torsion; its push-forward is the last term in the displayed exact sequence.

The \emph{focal adjoint line bundle} is
\[
\B_D:=\A_D\otimes \nu^*\mathcal O_D(-1).
\]
Equivalently, it is determined by the shifted adjoint sequence
\[
0\to \mathcal O_{(\bP^2)^\vee}(-2)
\to \mathcal I_{Z^{\rm ec}}(c-2)
\to \nu_*\B_D\to 0.
\]
\end{definition}

\begin{remark}\label{rem:nodal-cuspidal-adjoint}
In the nodal-cuspidal case the definition is completely explicit. If $p_i',p_i''\in \widetilde D$ are the two preimages of the node $p_i$, and if $q_j\in \widetilde D$ is the preimage of the cusp $q_j$, then
\[
\A_D=
\mathcal O_{\widetilde D}\left(
\nu^*\mathcal O_D(c-1)
-\sum_i(p_i'+p_i'')-3\sum_jq_j
\right),
\]
and
\[
\B_D=
\mathcal O_{\widetilde D}\left(
\nu^*\mathcal O_D(c-2)
-\sum_i(p_i'+p_i'')-3\sum_jq_j
\right).
\]
The correction divisor has degree $2\delta+3\kappa$, which is the Pl\"ucker drop in the class for nodes and cusps. Thus
\[
\deg \A_D=c(c-1)-2\delta-3\kappa=d,
\qquad
\deg \B_D=d-c.
\]
This is why one should not identify the symbol $\deg Z^{\rm ec}$ with the ordinary scheme-theoretic length of $Z^{\rm ec}$ in degree computations: the relevant number here is the adjoint divisor degree on the normalization.
\end{remark}

\begin{prop}\label{prop:kernel-B}
Let $D=\{G=0\}\subset (\bP^2)^\vee$ be a real reduced curve of degree $c$, genus $g$, and class $d$, corresponding to a smooth point of an equiclassical component $W\subset W_{c,g,d}(\bR)$. Assume that the usual equiclassical tangent-space description holds at $D$, that the equiclassical adjoint sequence holds, and that the equiclassical scheme of $D$ is disjoint from the reducible conic
\[
Y=\{u^2+v^2=0\}=L_+\cup L_-.
\]
Then the real tangent kernel of the focal map is naturally identified with the real part of
\[
H^0(\widetilde D,\B_D).
\]
Consequently,
\[
\dim_{\bR}\ker(d\Phi_G)=h^0(\widetilde D,\B_D)
\]
and, if $W$ is smooth of expected dimension at $G$,
\[
\operatorname{rank}(d\Phi_G)=c+d-g+1-h^0(\widetilde D,\B_D).
\]
\end{prop}

\begin{proof}
Work on the affine chart in which the coefficient of $w^c$ in $G$ is equal to $1$. By the tangent-space hypothesis, an infinitesimal real equiclassical deformation is represented by a real homogeneous polynomial $H$ of degree $c$ satisfying the linear equiclassical conditions, i.e. by a section of $H^0(\mathcal I_{Z^{\rm ec}}(c))$ with zero $w^c$-coefficient.

The condition that $H$ lie in the kernel of $d\Phi_G$ is $H(-1,-i,w)\equiv 0$. Since $H$ has zero $w^c$-coefficient in the normalized affine chart, this identity on the affine coordinate $w$ is the same as the projective condition $H|_{L_+}=0$, including the point at infinity of $L_+$. Since $H$ has real coefficients, this is equivalent to $H|_{L_-}=0$ as well. Hence $H$ vanishes on $Y=L_+\cup L_-$, whose equation is $u^2+v^2=0$. Therefore
\[
H=(u^2+v^2)Q
\]
for some homogeneous polynomial $Q$ of degree $c-2$.
Because $Z^{\rm ec}$ is disjoint from $Y$, the condition $H\in H^0(\mathcal I_{Z^{\rm ec}}(c))$ is equivalent to
\[
Q\in H^0(\mathcal I_{Z^{\rm ec}}(c-2)).
\]
Conversely, every such $Q$ gives an infinitesimal deformation in the kernel.
Thus the kernel is identified with the real part of $H^0(\mathcal I_{Z^{\rm ec}}(c-2))$.

By the shifted adjoint sequence in Definition~\ref{def:class-adjoint-bundle},
\[
0\to\mathcal O_{(\bP^2)^\vee}(-2)
\to \mathcal I_{Z^{\rm ec}}(c-2)
\to \nu_*\B_D\to0,
\]
and the vanishings
\[
H^0((\bP^2)^\vee,\mathcal O(-2))=H^1((\bP^2)^\vee,\mathcal O(-2))=0,
\]
taking cohomology gives
\[
H^0(\mathcal I_{Z^{\rm ec}}(c-2))\cong H^0(\widetilde D,\B_D).
\]
The dimension and rank formulas follow.
\end{proof}

\begin{thm}\label{thm:rational-rank}
Let $W\subseteq W_{c,0,d}(\bR)$ be an irreducible component of the real equiclassical family of rational curves in $(\bP^2)^\vee$ of degree $c$ and class $d$. Assume that $W$ is smooth of the expected dimension $c+d+1$ at a general point, that the usual equiclassical tangent-space description holds there, and that a general point of $W$ corresponds to a nodal-cuspidal curve $D\subset(\mathbb P^2)^\vee$ whose singular scheme is disjoint from $Y=\{u^2+v^2=0\}$. Thus the nodal-cuspidal hypothesis in this theorem is imposed on the dual curve $D$, not on its dual curve in the original plane. Then, for a general smooth point $G\in W$,
\[
\dim_{\bR}\ker(d\Phi_G)=\max(0,d-c+1)
\]
and
\[
\operatorname{rank}(d\Phi_G)=\min(2c,c+d+1).
\]
Equivalently, the expected local dimension of the confocal locus through $G$ is
\[
\max(0,d-c+1).
\]
\end{thm}

\begin{proof}
Let $D=\{G=0\}$ and let $\nu:\widetilde D\to D$ be the normalization. Since $D$ is rational, $\widetilde D\cong \mathbb P^1$. By Proposition~\ref{prop:kernel-B}, the kernel of the focal differential is the real part of $H^0(\widetilde D,\B_D)$.

In the nodal-cuspidal case, Remark~\ref{rem:nodal-cuspidal-adjoint} gives
\[
\deg \B_D=d-c.
\]
Therefore
\[
\B_D\cong \mathcal O_{\mathbb P^1}(d-c),
\]
and hence
\[
h^0(\widetilde D,\B_D)=h^0(\mathbb P^1,\mathcal O_{\mathbb P^1}(d-c))=
\max(0,d-c+1).
\]
Since the expected dimension of $W_{c,0,d}$ is $c+d+1$, the rank formula follows from Proposition~\ref{prop:kernel-B}:
\[
\operatorname{rank}(d\Phi_G)
=c+d+1-\max(0,d-c+1)
=\min(2c,c+d+1).
\]
\end{proof}

\begin{remark}\label{rem:rational-correction}
The maximal-rank conclusion in Theorem~\ref{thm:rational-rank} means maximal rank in the ordinary sense: the rank is the smaller of the source dimension and the target dimension. Thus it is not always equal to $2c$. For example, if $c>d+1$, then the source has dimension $c+d+1<2c$, so rank $2c$ is impossible.
\end{remark}

The preceding discussion shows that the cohomological object relevant to confocal deformations is $\B_D$, not $\A_D$ itself. The bundle $\A_D$ records the class adjunction, while the focal map asks for vanishing along the two conjugate isotropic lines and therefore subtracts one further hyperplane section.

\begin{conjecture}\label{conj:main}
Let $W$ be an irreducible component of the real equiclassical family of curves in $(\bP^2)^\vee$ of degree $c$, genus $g$, and class $d$, and let $G$ be a general smooth point of $W$. Put $D=\{G=0\}$. Assume that the usual equiclassical tangent-space description and adjoint sequences hold at $D$, that the equiclassical scheme of $D$ is disjoint from
\[
Y=\{u^2+v^2=0\},
\]
and that the focal adjoint bundle satisfies $\deg \B_D=d-c$; this holds, for instance, in the nodal-cuspidal case. Then
\[
h^0(\widetilde D,\B_D)=\max(0,d-g-c+1).
\]
If, in addition, $W$ is smooth of the expected real dimension $c+d-g+1$ at $G$, then Proposition~\ref{prop:kernel-B} implies that the focal map has maximal rank at $G$:
\[
\operatorname{rank}(d\Phi_G)=\min(2c,c+d-g+1).
\]
\end{conjecture}

\begin{prop}\label{prop:BN-reformulation}
Assume $\deg \B_D=d-c$, and set $b=d-c$. By Riemann--Roch, Conjecture~\ref{conj:main} is equivalent to the following Brill--Noether type alternative:
\begin{enumerate}
\item if $b\le g-1$, then $H^0(\widetilde D,\B_D)=0$;
\item if $b\ge g$, then $H^1(\widetilde D,\B_D)=0$.
\end{enumerate}
In particular, the conjectural maximal-rank statement is automatic whenever $b<0$, and the nonspeciality part is automatic whenever $b\ge 2g-1$.
\end{prop}

\begin{proof}
Riemann--Roch gives
\[
h^0(\B_D)-h^1(\B_D)=\deg \B_D-g+1=d-c-g+1=b-g+1.
\]
If $b\le g-1$, then $b-g+1\le 0$, and the expected value in Conjecture~\ref{conj:main} is $0$; this is exactly the assertion $H^0(\B_D)=0$. If $b\ge g$, then $b-g+1>0$, and the expected value is $b-g+1$; by Riemann--Roch this is equivalent to $H^1(\B_D)=0$.

If $b<0$, then a line bundle of negative degree has no nonzero sections, so the first alternative holds automatically. If $b\ge 2g-1$, then $\deg(K_{\widetilde D}\otimes \B_D^{-1})=2g-2-b<0$, hence $H^1(\B_D)=0$ by Serre duality.
\end{proof}

\begin{remark}\label{rem:old-A-not-enough}
The nonspeciality of the class adjoint bundle $\A_D$ alone is not sufficient for maximal rank of the focal map. It controls restriction to a single hyperplane section. The real focal condition fixes the restriction to both conjugate isotropic lines, and this is why the correct bundle is
\[
\B_D=\A_D\otimes \nu^*\mathcal O_D(-1).
\]
\end{remark}

\section{Known examples of confocal families}
In each example below we indicate explicitly whether prescribed real foci, or more precisely a prescribed focal divisor when multiplicities occur, determine a positive-dimensional confocal family, a finite set of curves, or a unique curve. Throughout, the dimension is the dimension of the parameter space of real curves with the prescribed focal data.

\subsection{Maximal-class rational curves} 
Let $c=2(d-1)$ general points in $\mathbb R^2$ be given. We ask for rational curves of
degree $d$ and class $c$ having these points as real foci. A simple dimension count gives
the answer.

The space of rational plane curves of degree $d$ has dimension
\[
\frac12 d(d+3)-\frac12 (d-1)(d-2)=3d-1.
\]
Prescribing $c=2(d-1)$ real foci imposes
\[
2c=4(d-1)
\]
real conditions. Hence the expected dimension is
\[
(3d-1)-4(d-1)=3-d.
\]

Therefore:
\begin{itemize}
\item if $d=2$, the expected dimension is $1$, so one obtains a one-dimensional confocal family;
\item if $d=3$, the expected dimension is $0$, so one obtains a finite set of curves for general prescribed foci;
\item if $d\ge 4$, the expected dimension is negative, so for general prescribed foci no such rational curve exists.
\end{itemize}
In particular, among maximal-class rational curves, nontrivial confocal families occur only for conics. Rational curves with many cusps can have smaller class and are covered by Theorem~\ref{thm:rational-rank}.

\subsection{Curves with given foci of minimal class}
The following proposition is essentially stated in \cite{Hi}*{Ch.~5, §1, p.~69}.

\begin{prop}
Let $c\ge 2$, let $P_j=(x_j:y_j:1)\in \bP^2(\bR)$, $j=1,\dots,c$, be distinct points, and put $z_j=x_j+iy_j$. Then there exists a real plane curve of class $c$ whose finite real foci are precisely $P_1,\dots,P_c$. Moreover, class $c$ is the minimal possible. More precisely, the normalized dual equations with these prescribed foci form the affine space of dimension $\binom{c}{2}$ given by
\[
G(u,v,w)=\prod_{j=1}^c (x_ju+y_jv+w)+(u^2+v^2)Q_{c-2}(u,v,w),
\]
where $Q_{c-2}$ is an arbitrary real homogeneous polynomial of degree $c-2$. A nonempty Zariski open subset of this affine space consists of reduced irreducible dual curves of degree $c$; the duals of these curves are real plane curves of class $c$ with the prescribed focal divisor.
\end{prop}

\begin{proof}
Set
\[
G_0(u,v,w):=\prod_{j=1}^c (x_ju+y_jv+w).
\]
Then
\[
G_0(-1,-i,w)=\prod_{j=1}^c(w-x_j-iy_j)=\prod_{j=1}^c(w-z_j),
\]
so $G_0=0$ has the prescribed foci. If
\[
G(u,v,w)=G_0(u,v,w)+(u^2+v^2)Q_{c-2}(u,v,w),
\]
then $G$ is still a real homogeneous polynomial of degree $c$, and since $u^2+v^2=0$ on the isotropic lines $u \pm iv=0$, we get
\[
G(-1,-i,w)=G_0(-1,-i,w)=\prod_{j=1}^c(w-z_j).
\]
Hence every such $G$ has the prescribed finite focal divisor. Let
\[
\Lambda=G_0+(u^2+v^2)H^0((\mathbb P^2)^\vee,\mathcal O(c-2))
\]
be this affine linear system. Its base locus is contained in
\[
Y\cap\{G_0=0\},
\]
which consists of the $2c$ simple points on the two isotropic lines corresponding to the distinct prescribed foci. Away from this finite base locus the moving part $(u^2+v^2)Q_{c-2}$ separates first jets in the usual Bertini sense, so a general member is smooth there. At a base point on $L_+$, the restriction $G_0|_{L_+}$ has a simple zero, because the prescribed numbers $z_j=x_j+iy_j$ are distinct; hence the derivative tangent to $L_+$ is nonzero. The same argument applies on $L_-$. Thus a general member is nonsingular also along the base locus. A smooth plane curve is reduced and irreducible, since two distinct positive-degree components would meet by Bezout and create a singular point. Therefore a nonempty Zariski open subset of choices of $Q_{c-2}$ gives reduced irreducible curves $G=0$ of degree $c$.  The dual of such a general member is therefore an irreducible real curve of class $c$ with the prescribed foci. The space of choices of $Q_{c-2}$ has (affine) dimension
\[
\dim H^0((\mathbb P^2)^\vee,\mathcal O_{(\mathbb P^2)^\vee}(c-2))=\frac{1}2 c(c-1)=\binom{c}2.
\]
Finally, if a curve has class $c'$, then its focal polynomial $G(-1,-i,w)$ has degree $c'$, so it has at most $c'$ real foci counted with multiplicity. Therefore $c$ distinct prescribed foci force $c'\ge c$, and the above construction shows that $c'=c$ is attainable.
\end{proof}

Thus, for arbitrary distinct prescribed real foci, curves of minimal class always exist.
Their normalized dual equations form an affine space of dimension $\binom{c}{2}$, and the actual reduced irreducible class-$c$ curves occupy a nonempty open subset of this space. In particular, for $c\ge 2$
the foci do not determine a unique curve of class $c$ with the given foci. 

\begin{example}[Conics with two prescribed real foci]
Let $P_j=(x_j:y_j:1)$, $j=1,2$, be two points in $\bP^2(\bR)$. Then the real conics with
real foci $P_1$ and $P_2$ form a one-dimensional confocal family. In particular, the foci
do \emph{not} determine a unique conic.

Let
\[
G(u,v,w)=w^2+\alpha_1wv+\alpha_2wu+\alpha_3v^2+\alpha_4vu+\alpha_5u^2
\]
be the equation of the dual conic. The condition that $P_1$ and $P_2$ be the real foci is
\[
G(-1,-i,w)=(w-x_1-iy_1)(w-x_2-iy_2).
\]
Expanding both sides and comparing coefficients gives
\[
\alpha_1=y_1+y_2,\qquad
\alpha_2=x_1+x_2,\qquad
\alpha_4=x_1y_2+x_2y_1,
\]
and
\[
-\alpha_3+\alpha_5=x_1x_2-y_1y_2.
\]
Equivalently,
\[
\alpha_3=\alpha_5-x_1x_2+y_1y_2.
\]
Hence four independent real conditions are imposed on the five parameters
$\alpha_1,\dots,\alpha_5$, leaving exactly one free real parameter.

Therefore the dimension is
\[
5-4=1.
\]
So the prescribed foci determine a one-parameter confocal family of real conics.

By choosing $x_1=1$, $x_2=-1$, $y_1=y_2=0$, we recover the example given by Emch in
\cite{Emch}*{5., p.~160}.
\end{example}

\begin{example}[Curves of class $3$ with three prescribed real foci]
Consider real curves of class $3$, given by the equation of the dual curve
\begin{multline*}
G(u,v,w)=w^3+\alpha_1w^2v+\alpha_2w^2u+\alpha_3wv^2+\alpha_4wvu+\alpha_5wu^2\\
+\alpha_6v^3+\alpha_7v^2u+\alpha_8vu^2+\alpha_9u^3.
\end{multline*}
Fix three real foci $P_j=(x_j:y_j:1)$, $j=1,2,3$. Then the real curves of class $3$ with these
prescribed real foci form a three-dimensional confocal family. Again, the foci do
\emph{not} determine a unique curve.

Indeed, the focal condition is
\[
G(-1,-i,w)=\prod_{j=1}^3(w-x_j-iy_j).
\]
Now
\begin{multline*}
G(-1,-i,w)=w^3-(i\alpha_1+\alpha_2)w^2+(-\alpha_3+i\alpha_4+\alpha_5)w\\
+(i\alpha_6+\alpha_7-i\alpha_8-\alpha_9).
\end{multline*}
Comparing coefficients with
\[
\prod_{j=1}^3(w-x_j-iy_j)
\]
gives three complex equations, hence six real linear conditions on the nine real
parameters $\alpha_1,\dots,\alpha_9$. More explicitly,
\[
\alpha_1=y_1+y_2+y_3,\qquad
\alpha_2=x_1+x_2+x_3,
\]
while the combinations
\[
\alpha_3-\alpha_5,\qquad \alpha_4,\qquad \alpha_6-\alpha_8,\qquad \alpha_7-\alpha_9
\]
are also determined by the prescribed foci.

Thus exactly three real parameters remain free, so the dimension is
\[
9-6=3.
\]
If the curve $G(u,v,w)=0$ is nonsingular, then $G\in W_{3,1,6}$, and the corresponding dual curves are sextics with nine cusps. If $G(u,v,w)=0$ is nodal, then $G\in W_{3,0,4}$, and there is a 2-dimensional set of corresponding confocal tri-cuspidal quartics. If $G(u,v,w)=0$ is cuspidal, then $G\in W_{3,0,3}$, and there is a 1-dimensional set of corresponding confocal cuspidal cubics. According to Roberts \cite{Roberts1901}*{p.~154}, in this last case there are 36 different one-dimensional families.
\end{example}

\subsection{Siebeck curves}
Siebeck \cite{Siebeck} showed that if $f\in \mathbb C[z]$ is a cubic polynomial whose
roots $z_1,z_2,z_3$ are not collinear (viewed as points in $\mathbb R^2$), then the roots
of $\partial f/\partial z$ are the foci of the unique conic tangent to the sides of the
triangle with vertices $z_1,z_2,z_3$ at their midpoints.

Thus, once the triangle $(z_1,z_2,z_3)$ is fixed, the corresponding \emph{Siebeck conic} is
\emph{unique}; equivalently, the parameter space has dimension $0$ 
for this fixed tangency problem.

However, this should not be confused with uniqueness from the foci alone. For two
prescribed real foci, the space of all real conics with those foci is one-dimensional, as
in the conic example above. The Siebeck construction therefore produces a distinguished
subclass of that one-dimensional confocal family, but the foci by themselves do not
determine a unique conic.

So, in the language of this paper:
\begin{itemize}
\item for fixed triangle data, there is a unique Siebeck conic (dimension $0$);
\item for fixed foci alone, there is a one-dimensional confocal family of conics.
\end{itemize}

Linfield \cite{Linfield}*{Thm.~1, p.~247} extended Siebeck's result to arbitrary degree, see also
\cite{CA2}*{Thm.~5.1, p.~235}. Let $f\in \mathbb C[z]$ be a polynomial of degree $n$ with roots
$z_j=x_j+iy_j$, $j=1,\dots,n$, and let
\[
H(u,v,w):=\prod_{j=1}^n (x_ju+y_jv+w).
\]
Then $H=0$ is the union of the lines in $(\mathbb P^2)^\vee$ dual to the points $z_j$,
and the polar curve of $H=0$ with respect to $(0:0:1)\in (\mathbb P^2)^\vee$ is given by
\[
\frac{\partial H}{\partial w}=0.
\]
The curve $\partial H/\partial w=0$ is dual to a curve of class $n-1$, whose foci are the
roots of $\partial f/\partial z=0$.

\begin{example}[Siebeck curves of class $n-1$]
Casas-Alvero \cite{CA1,CA2} developed this point of view further. A special case of
\cite{CA2}*{Thm.~6.1, p.~236} is the following. Let $f\in \mathbb C[z]$ be a polynomial
of degree $n$ with distinct roots $z_1,\dots,z_n$, and let $p_{j,k}$ be the midpoint of
the segment joining $z_j$ and $z_k$ in $\mathbb R^2$. Then there is a unique curve $S$,
the \emph{Siebeck curve}, of class $n-1$, tangent to each line $z_jz_k$ at $p_{j,k}$.
Its foci are the roots of $\partial f/\partial z=0$.

Again, the important point is that the curve is unique \emph{after the points
$z_1,\dots,z_n$ have been fixed}, or equivalently after all the tangency conditions have
been prescribed. For this incidence/tangency problem the parameter space therefore has
dimension $0$.

By contrast, if one fixes only the $n-1$ foci and considers \emph{all} real curves of
class $n-1$ with those foci, then by the above results the ambient confocal
family has dimension
$\binom{n-1}{2}.$ 

Thus the Siebeck curve is not uniquely determined by the foci alone in general; rather,
the Siebeck construction selects a distinguished member (or distinguished subclass) inside
that larger confocal family by imposing additional tangency conditions.

So, in this example:
\begin{itemize}
\item for fixed roots $z_1,\dots,z_n$ (equivalently, fixed tangency data), there is a unique Siebeck curve (dimension $0$);
\item for fixed foci alone, the full space of curves of class $n-1$ has dimension $\binom{n-1}{2}$.
\end{itemize}
\end{example}

\subsection{Poncelet curves}
Hunziker et al.\ \cite{Hunziker} study algebraic curves arising as envelopes of polygons
supported on the unit circle $\mathbb T\subset \mathbb C=\mathbb R^2$. In
\cite{Hunziker}*{Def.~3.2, p.~15} they define a family of $n$-Poncelet polygons to be a
family of $n$-gons $\mathcal P(z)$, $z\in\mathbb T$, inscribed in $\mathbb T$, such that
$z$ is a vertex of $\mathcal P(z)$ and such that, whenever $w\in\mathbb T$ is also a
vertex of $\mathcal P(z)$, one has $\mathcal P(w)=\mathcal P(z)$.

An \emph{$n$-Poncelet curve} is then a closed curve in the unit disc $\mathbb D$ which
envelopes such a family of $n$-Poncelet polygons. In particular, the curve determines the
corresponding family of polygons uniquely.

The authors also define \emph{complete} Poncelet curves \cite{Hunziker}*{Def.~3.11,
p.~21}. A complete $n$-Poncelet curve \emph{of minimal class} gives a real algebraic curve
$C$ of class $n-1$, all of whose real foci lie in $\mathbb D$. In this case there is
a bijection between configurations of $n-1$ points in $\mathbb D$ and complete
$n$-Poncelet curves, and $C$ can be reconstructed from its real foci.

Therefore, in the minimal-class complete case, prescribed real foci determine a
\emph{unique} Poncelet curve. Equivalently, the parameter space has dimension
$0.$ 
So this is \emph{not} a positive-dimensional confocal family.

The authors also ask whether the assumption that the curve has class $n-1$ is in fact
superfluous \cite{Hunziker}*{Rmk.~4.2, p.~29}. Thus, beyond the minimal-class complete
case, uniqueness from the foci is not asserted. 

\section*{Appendix. Comments on Conjecture~\ref{conj:main} by Eugenii Shustin.}



We demonstrate that Conjecture~\ref{conj:main} holds in a range which is not covered by Proposition~\ref{prop:BN-reformulation}.
 We use notations from Conjecture~\ref{conj:main}. Our goal is the following statement.


\begin{prop}\label{p1}
In the setting of Conjecture~\ref{conj:main}, the following holds:
\begin{enumerate}\item[(1)] If $c+d\ge2g$, then
a generic curve $D\in W$ is nodal-cuspidal.
\item[(2)] If, in addition,
\begin{equation}4n+9k\le(c+1)^2,\label{e4}\end{equation}
where $n$ and $k$ are the numbers of nodes and cusps of $D$, respectively, then
\begin{equation}H^1(\widetilde D,{\mathcal B}_D)=0.\label{e3}\end{equation}
\end{enumerate}
\end{prop}

\begin{remark}\label{r1}
For a nodal-cuspidal curve of degree $c$, condition $b\ge2g-1$ from Proposition~\ref{prop:BN-reformulation} is equivalent to $k\le c-1$, while conditions from Proposition \ref{p1} are equivalent to
$$k\le3c-2\quad\text{and}\quad 4n+9k\le(c+1)^2.$$
\end{remark}

\begin{proof}[Proof of Proposition~\ref{p1}(1)]
We argue on the contrary assuming that a generic member $D\in W$ has singular points different from ordinary nodes and cusps. The germ $(W,D)$ is an equisingular family, and its Zariski tangent space can be identified with $$H^0(D,{\mathcal J}_{Z^{es}(D)/D}(c)),$$
where ${\mathcal J}_{Z^{es}(D)/D}(c)\subset{\mathcal O}_D\otimes{\mathcal O}_{\PP^2}(c)$ is the twisted ideal sheaf of the equisingular zero-dimensional scheme $Z^{es}(D)\subset D$ associated with $D$ (see, for example, \cite[Section 2.1]{GK}, \cite[Section 2.2.2]{GLS1}, and \cite[Section 2.2.1.2]{GLS2}). By \cite[Lemma~14]{Sh24}, $Z^{ec}(D)\subsetneq Z^{es}(D)$. Take a zero-dimensional scheme $Z\subset D$ such that $Z^{ec}(D)\subsetneq Z\subset Z^{es}(D)$ and $\deg Z=\deg^{ec}(D)+1=\varkappa(D)-\delta(D)+1$ (where $\varkappa(D)$ and $\delta(D)$ are the total $\varkappa$-invariant and $\delta$-invariant of all singular points of $D$).

We will prove that
\begin{equation}H^1(D,{\mathcal J}_{Z/D}(c))=0.\label{e1}\end{equation}
If it is so, we derive
$$\dim W\le h^0(D,{\mathcal J}_{Z^{es}(D)/D}(c))\le
h^0(D,{\mathcal J}_{Z/D}(c))$$
$$=h^0(D,{\mathcal O}_D(c))-\deg Z<h^0(D,{\mathcal O}_D(c))-\deg Z^{ec}(D)$$
$$=\frac{c(c+3)}{2}-\varkappa(D)+\delta(D),$$
which contradicts the fact that
$$\dim W\ge\frac{c(c+3)}{2}-\deg Z^{ec}(D)=\frac{c(c+3)}{2}-\varkappa(D)+\delta(D).$$

To prove (\ref{e1}), we apply \cite[Proposition 5.2(i)]{GK}. In our situation, the sufficient condition for (\ref{e1}) reads
\begin{equation}\chi({\mathcal J}_{Z/D}(c))>\chi(\omega_D)+\ind_D({\mathcal J}_{Z/D}(c),\omega_D),\label{e2}\end{equation}
where $\omega_D$ is the dualizing sheaf, and the parameter $\ind_D({\mathcal F},\omega_D)$ is defined in \cite[Section 5.1]{GK}. Here
\[
\chi({\mathcal J}_{Z/D}(c))=\chi({\mathcal O}_D(c))-\deg Z
=\frac{c(c+3)}{2}-\deg Z=c+d-g,
\]
because $\deg Z=\varkappa(D)-\delta(D)+1$ and $d=c(c-1)-\varkappa(D)$. Also
\[
\chi(\omega_D)=p_a(D)-1=\frac{(c-1)(c-2)}2-1=\frac{c(c-3)}2,
\]
and
\[
\ind_D({\mathcal J}_{Z/D},\omega_D)
\le\ind_D({\mathcal J}^{cond}_D,\omega_D)=-\delta(D),
\]
since ${\mathcal J}_{Z/D}$ is a subsheaf of the conductor ideal sheaf ${\mathcal J}^{cond}_D$; see also \cite[Example in Section 5.1]{GK}. Therefore
\[
\chi(\omega_D)+\ind_D({\mathcal J}_{Z/D}(c),\omega_D)
\le p_a(D)-1-\delta(D)=g-1.
\]
The numerical hypothesis $c+d\ge2g$ gives
\[
\chi({\mathcal J}_{Z/D}(c))=c+d-g\ge g>g-1,
\]
so (\ref{e2}) holds, and hence (\ref{e1}) follows.
\end{proof}

\begin{proof}[Proof of Proposition~\ref{p1}(2)]
We can suppose that $D$ is nodal-cuspidal and $c\ge4$. In this case, $Z^{ec}(D)=Z^{es}(D)=Z^{ea}(D)$ and the shifted adjoint sequence gives
\[
\nu_*{\mathcal B}_D={\mathcal J}_{Z^{ea}(D)/D}(c-2),
\qquad
H^1(\widetilde D,{\mathcal B}_D)\simeq
H^1(D,{\mathcal J}_{Z^{ea}(D)/D}(c-2)).
\]
Here ${\mathcal J}_{Z^{ea}(D)/D}$ denotes the ideal sheaf of $Z^{ea}(D)$ on $D$, while ${\mathcal J}_{Z^{ea}(D)/\PP^2}$ denotes the corresponding ideal sheaf on the ambient plane. Note also that by \cite[Theorem 3.5.7]{GLS2},
\begin{equation}
H^1(\PP^2,{\mathcal J}_{Z^{ea}(D)/\PP^2}(c))=0.
\label{e5}
\end{equation}

To establish the $h^1$-vanishing (\ref{e3}), we closely follow the lines of the proof of \cite[Theorem 3.5.7]{GLS2}.

We argue by contradiction. Assume that
\[
h^1(D,{\mathcal J}_{Z^{ea}(D)/D}(c-2))>0.
\]
From the exact sequence
\[
0\to{\mathcal O}_{\PP^2}(-2)\overset{\cdot D}{\to}
{\mathcal J}_{Z^{ea}(D)/\PP^2}(c-2)\to
{\mathcal J}_{Z^{ea}(D)/D}(c-2)\to0
\]
and the vanishings $H^1(\PP^2,{\mathcal O}_{\PP^2}(-2))=H^2(\PP^2,{\mathcal O}_{\PP^2}(-2))=0$, we derive that
\[
h^1(\PP^2,{\mathcal J}_{Z^{ea}(D)/\PP^2}(c-2))>0.
\]
Let $Z\subset Z^{ea}(D)$ be a minimal subscheme satisfying $h^1(\PP^2,{\mathcal J}_{Z/\PP^2}(c-2))>0$. Clearly, $Z$ is a local complete intersection. We have
$$\deg Z\le n+2k<\frac{4n+9k}{4}\le\frac{(c+1)^2}{4}<\frac{(c-1)(c+2)}{2},$$
and hence there exists a curve of degree $<c$ containing $Z$. Furthermore, by \cite[Lemma 3.5.6]{GLS2} applied under the extra assumption (\ref{e5}), there exists a curve $\Gamma_m$ of degree $m\le\frac{c+1}{2}$ such that
$$h^1(\PP^2,{\mathcal J}_{Z\cap \Gamma_m/\PP^2}(c-2))=h^1(\PP^2,{\mathcal J}_{Z/\PP^2}(c-2))>0$$
and
\begin{equation}\deg(Z\cap \Gamma_m)\ge m(c+1-m)\label{e6}\end{equation}
(comparing with \cite[Lemma 3.5.6]{GLS2}, we replace the term "$d+3$" with "$c+1$" and use \cite[Figure 3.4]{GLS2} in which the right-most limit of the graph is $c$, due to (\ref{e5})). By the minimality assumption, $Z\cap \Gamma_m=Z$. Thus, inequality (\ref{e6}) implies
\begin{equation}m\le\frac{2\cdot\deg Z}{c+1+\sqrt{(c+1)^2-4\cdot\deg Z}}.\label{e7}\end{equation}
This is an analogue of \cite[Formula (3.5.0.8)]{GLS2}.
Then we repeat word-for-word the computations in the proof of \cite[Theorem 3.5.7]{GLS2} that follow the latter cited inequality
\cite[Formula (3.5.0.8)]{GLS2}, provided, that in these computations, we replace "$d+3$" with "$c+1$" and "$k$" with "$m$". In the very end, we arrive to the inequality
\begin{equation}(c+1)^2<\sum_{i=1}^r\left(\frac{(\deg Z''_i)^2}{\Delta_i}+2\cdot\deg Z''_i+\Delta_i\right),\label{e8}\end{equation}
where $Z''$ is some zero-dimensional subscheme of $Z$, $Z''=Z''_1\cup...\cup Z''_r$ is the decomposition into irreducible components, and $1\le\Delta_i\le\deg Z''_i$ for all $i=1,...,r$ (cf. the definition of $\Delta_i$ in \cite[Page 316]{GLS2}). Hence, (\ref{e8}) yields
$$(c+1)^2<\max\left\{\sum_{i=1}^r(\deg Z''_i+1)^2,\ 4\cdot\deg Z''\right\}\le4n+9k,$$
which contradicts the assumption (\ref{e4}).
\end{proof}

\end{document}